\documentclass{amsart}
\usepackage{url,hyperref}
\usepackage[all]{xy}

\newtheorem{theorem}{Theorem}
\newtheorem{proposition}[theorem]{Proposition}

\theoremstyle{definition}

\newtheorem{example}[theorem]{Example}
\newtheorem*{example*}{Example}
\newtheorem{remark}[theorem]{Remark}

\newcommand\qbin[3]{\left[\begin{matrix} #1 \\ #2 \end{matrix} \right]_{#3}}

\newcommand\multichoose[2]{\left(\!\!\! \left( \begin{matrix} #1 \\ #2 \end{matrix} 
\right)\!\!\! \right)}
\newcommand\bin[2]{\left( \begin{matrix} #1 \\ #2 \end{matrix} 
\right)}

\newcommand\Sym{\operatorname{Sym}}

\newcommand\inv{\operatorname{inv}}
\newcommand\maj{\operatorname{maj}}

\newcommand\Gal{\operatorname{Gal}}

\newcommand\reg{\operatorname{reg}}
\newcommand\one{\mathbf{1}}

\newcommand\Symm{\mathfrak{S}}

\def\<{{\langle}}
\def\>{{\rangle}}
\newcommand\xx{{\mathbf{x}}}
\newcommand\yy{{\mathbf{y}}}
\newcommand\uu{{\mathbf{u}}}

\newcommand\dd{{\mathbf{d}}}
\newcommand\cc{{\mathbf{c}}}

\newcommand\CCC{{\mathbf{C}}}
\newcommand\stat{{\mathbf{s}}}

\newcommand\ZZ{{\mathbb Z}}
\newcommand\FF{{\mathbb F}}
\newcommand\QQ{{\mathbb Q}}
\newcommand\NN{{\mathbb N}}
\newcommand\CC{{\mathbb C}}

\title{Constructions for cyclic sieving phenomena}

\author{Andrew Berget}
\email{berget@math.ucdavis.edu}
\address{Mathematical Sciences Building\\
One Shields Ave.\\
University of California\\
Davis, CA 95616\\
USA}
\author{Sen-Peng Eu}
\email{speu@nuk.edu.tw}
\address{
No. 700 National University of Kaoshiung Rd.\\ Nanyang District, Kaoshiung 811\\Taiwan
}
\author{Victor Reiner}
\email{reiner@math.umn.edu}
\address{School of Mathematics\\
University of  Minnesota\\
Minneapolis, MN 55455\\ USA}

\date{\today}

 \thanks{A. Berget is partially supported by NSF
  grant DMS-0636297. S.-P. Eu is partially supported by the National
  Science Council of Taiwan under grant NSC
  98-2115-M-390-002-MY3. V. Reiner was supported by NSF grant
  DMS-0601010.}



\begin{document}

\begin{abstract}
  We show how to derive new instances of the cyclic sieving phenomenon
  from old ones via elementary representation theory.  Examples are
  given involving objects such as words, parking functions, finite
  fields, and graphs.
\end{abstract}

\maketitle

%%%%%%%%%%%%%%%%%%%%%%%%%%%%%%%%%%%%%%%%%%%%%%%%%%%%%%%%%%%%%%%%%%%
\section{Introduction}
\label{sec:intro}

The {\it cyclic sieving phenomenon (CSP)} was 
introduced in \cite{RSW}, generalizing Stembridge's 
{\it $q=-1$ phenomenon} \cite{Stembridge}.  The CSP pertains to a 
finite set $X$, carrying
the permutation action of a finite abelian group 
written explicitly as a product
$\CCC:=C_1 \times \cdots \times C_m$ of cyclic groups $C_i$, and
a polynomial $X(\uu):=X(u_1,\ldots,u_m)$ in $\ZZ[\uu]$,
often a generating function for the elements of $X$
according to some natural statistic(s).
One says that the triple $(X,X(\uu),\CCC)$
{\it exhibits the CSP}
if after choosing embeddings\footnote{See Appendix~\ref{appendix}
below for a discussion of the dependence on this choice
of embeddings, and on the choice of decomposition
$\CCC:=C_1 \times \cdots \times C_m$.}
of groups 
$\omega_i:C_i \hookrightarrow \CC^\times$,
one has for every $\cc=(c_1,\ldots,c_m)$ in $\CCC$ that the cardinality of
its fixed point set $X^{\cc}:=\{x \in X: \cc(x)=x\}$ is given by
$$
|X^\cc| = \left[ X(\uu) \right]_{u_i=\omega_i(c_i)}.
$$
In other words, the generating function $X(\uu)$ not only
has the usual property that its evaluation with all $u_i=1$
gives the cardinality $|X|$, but furthermore, its evaluation
at appropriate roots-of-unity carries all the numerical information
about the $\CCC$-orbit structure on $X$.  

For example, one can equivalently rephrase 
the CSP (see \cite[Proposition 3.1]{BarceloReinerStanton})
as a combinatorial interpretation for the coefficients in
the unique expansion
$$
X(\uu) \equiv 
  \sum_{\substack{\dd=(d_1,\ldots,d_m)\\ 0 \leq d_i <|C_i|}} a_\dd \,\, \uu^\dd
 \mod (u_1^{|C_1|}-1,\ldots,u_m^{|C_m|}-1).
$$
Specifically, the CSP asserts that the constant term $a_{(0,0,\ldots,0)}$ 
counts the total number of $\CCC$-orbits on $X$, 
and more generally, the coefficient $a_\dd$ counts the number of
$\CCC$-orbits for which the 
pointwise-$\CCC$-stabilizer subgroup of any element in the
orbit lies in the kernel of the degree one character
\begin{equation}
\label{degree-one-characters}
\omega^\dd:=\prod_{i=1}^m \omega_i^{d_i}.
\end{equation}

In \cite{BarceloReinerStanton, RSW}, various
instances of CSP's were shown, sometimes proven
via representation theory.
The point of the current paper is to show how this viewpoint,
combined with the standard multilinear constructions from
representation theory of
{\it tensor products} $V_1 \otimes V_2$,
{\it symmetric powers} $\Sym^k(V)$,
{\it exterior powers} $\bigwedge^k(V)$,
and {\it tensor powers} $V^{\otimes \ell}:=V \otimes \cdots \otimes V$,
allow one to automatically construct new CSP's from old ones.
Section~\ref{sec:constructions} develops these constructions, and uses them
to derive some interesting new CSP's.  We remark that a somewhat different
use of representation theoretic constructions to derive new CSP's
appears in Westbury \cite{Westbury}.

We illustrate our results in the remainder of this introduction, including some explicit
examples.   For the sake of stating these, recall a notion 
from \cite{RSW}: A cyclic group 
acts {\it nearly freely} on a finite set if either all orbits have the same
size, or if there is a unique singleton orbit and all non-singleton orbits 
have the same size.  We will also need a few notations.  For a positive integer $n$, define
$$
\begin{aligned}[]
[n]&:=\{1,2,\ldots,n\}\\
[n]_u&:=1+u+u^2+\cdots+u^{n-1}
\end{aligned}
$$
and for a polynomial $f(x_1,\ldots,x_n)$
in a variable set $\xx=(x_1,\ldots,x_n)$,
its {\it principal $u$-specialization} is $f(1,u,u^2,\ldots,u^{n-1})$.

\subsection{Words}

Consider the set $[n]^\ell$ of
words $w_1 w_2 \cdots w_\ell$ of length $\ell$
with letters in the alphabet $[n]$.
Given such a word $w$, its {\it inversion number} $\inv(w)$
is the number of pairs $(i,j)$ with $1 \leq i < j \leq \ell$ 
for which $w(i) > w(j)$, while its {\it major index} $\maj(w)$ 
is the sum of all positions $i$ in the range $1 \leq i < n$ for
which $w(i) > w(i+1)$.
A famous result of MacMahon (see \cite{FoataSchutzenberger}) asserts that these two statistics
are equidistributed as one runs over all rearrangements of a fixed word, so that
one has an equality
\begin{equation}
\label{words-maj-inv}
f(\xx,t):=\sum_{w \in [n]^\ell} x_{w_1} \cdots x_{w_\ell} \,\, t^{\maj(w)}
=\sum_{w \in [n]^\ell} x_{w_1} \cdots x_{w_\ell} \,\, t^{\inv(w)}.
\end{equation}

\begin{theorem}
\label{thm:words-biCSP}
Let $X=[n]^\ell$ be permuted by $C_1 \times C_2$
in which $C_1$ is a cyclic group acting nearly freely on the
letter values $[n]$, and $C_2$ is a cyclic group acting nearly freely on the word
positions $[\ell]$.  

Let $X(u,t)$ be the principal $u$-specialization
in the $\xx$-variables of $f(\xx,t)$.

Then $(X,X(u,t),C_1 \times C_2)$ exhibits the CSP.
\end{theorem}

%\vskip.1in
%\noindent
%{\bf Example.}
\begin{example*}
Take $n=3$, $\ell=2$, so that $X=[3]^2$.  Let $C_1=\langle c_1
\rangle$ be a cyclic group of order $3$ cyclically permuting the
letter values $[3]$, and let $C_2=\langle c_2 \rangle$ be a cyclic
group of order $2$ swapping the two positions $[2]$ in the words.
Then the set $X=[3]^2$ decomposes into these $C_1 \times C_2$-orbits
$$
\left\{
\begin{matrix}
12& 23& 31\\
u &u^3&u^2t\\
21& 32& 13\\
ut &u^3t&u^2\\
\end{matrix}
\right\}
,\quad
\left\{
\begin{matrix}
11&22&33\\
1 &u^2&u^4\\
\end{matrix}
\right\}
$$
in which each element of $X$ is shown with the term it contributes
to $X(u,t)$ just below it. The orbits are arranged tabularly 
so that $C_1, C_2$ act cyclically on the row, column indices
respectively.  The first orbit is $C_1 \times C_2$-free,
while in the second orbit $c_2$ acts trivially.

From the data above (or see Section~\ref{sec:hook-formulas} below) one can calculate
\[
X(u,t) = 1+u+2u^2+u^3+u^4
        +t(u+u^2+u^3)
\]
and hence
$$
X(u,t) \equiv 2+2u+2u^2 +t(1+u+u^2) \mod (u^3-1,t^2-1).
$$
Note that in this last expression,
the constant term $2$ matches the total number of orbits.
As an example of the root-of-unity evaluations predicted by the CSP,
note that $X(e^{\frac{2\pi i}{3}},1)=X(e^{\frac{2\pi i}{3}},-1)=0$, 
corresponding to the fact that neither
$(c_1,1)$ nor $(c_1,c_2)$ fix any elements of $X$.  On the other hand,
$X(1,-1)=3$ counts the
elements in the second orbit, which are fixed by $(1,c_2)$.  
\end{example*}

\subsection{Finite fields}
Theorem~\ref{thm:words-biCSP} combined with the
Normal Basis Theorem from Galois theory will turn out 
to have the following consequence for the action of the Frobenius endomorphism on a 
finite field $\FF_{q^\ell}$ for any prime power $q$.

\begin{theorem}
\label{thm:finite-field-biCSP}
Let $X=\FF_{q^\ell}$ be permuted by $C_1 \times C_2$
in which the cyclic group $C_1=\FF_q^\times$ of order $q-1$
acts via multiplication, and the cyclic group $C_2=\Gal(\FF_{q^\ell}/\FF_q)$ of
order $\ell$ generated by the Frobenius endomorphism acts as usual.

Let $X(u,t)$ be the same as in Theorem~\ref{thm:words-biCSP}, taken
with $n:=q$.

Then $(X,X(u,t),C_1 \times C_2)$ exhibits the CSP.
\end{theorem}

%\vskip.1in
%\noindent
%{\bf Example.}
\begin{example*}
Take $q=3$, $\ell=2$, so that 
$$
X=\FF_{3^2}=\FF_9=\{0,1,\pi,\pi^2,\pi^3,\pi^4,\pi^5,\pi^6,\pi^7\}
$$
where $\pi$ is a cyclic generator for the multiplicative group
$\FF_9^\times \cong \ZZ_8$.  The subfield $\FF_3$ embeds in $\FF_9$ as
$\FF_3=\{0,1,\beta\}$ where $\beta=\pi^4$, and
$C_1=\FF_3^\times=\langle \beta \rangle=\{1,\beta\}$ is a cyclic group
of order $2$ acting on $X$ by multiplication.  The Frobenius map
$F:\alpha \mapsto \alpha^3$ generates the Galois group
$C_2=\Gal(\FF_{3^2}/\FF_3)=\langle F \rangle$ of order two, also
acting on $X$.  Then the set $X$ decomposes into these three $C_1
\times C_2$-orbits
$$
\left\{
\begin{matrix}
\pi& \overset{F}{\longleftrightarrow}& \pi^3\\
\beta\updownarrow& &\updownarrow \beta\\
\pi^5& \overset{F}{\longleftrightarrow}& \pi^7\\
\end{matrix}
,\qquad
\begin{matrix}
&\pi^2& \\
\beta&\updownarrow &F\\
&\pi^6& \\
\end{matrix}
,\qquad
\xymatrix{
1 \ar@(ul,dl)[]|{F} 
\ar@{<->}^\beta[r] &  
\beta \ar@(ur,dr)[]|{F} 
}
,\qquad
\xymatrix{
0 \ar@(ul,dl)[]|{F} \ar@(ur,dr)[]|{\beta} 
}
\right\}
$$
The first orbit is $C_1 \times C_2$-free.  The second orbit has both
$\beta$ and $F$ acting by swapping the two elements.  The third orbit is
fixed by $F$ and has its two elements swapped by $\beta$.  The last
orbit is a singleton fixed by both $\beta$ and by $F$.

In the example following Theorem~\ref{thm:words-biCSP} we computed
$$
\begin{aligned}
X(u,t) &= 1+u+2u^2+u^3+u^4 +t(u+u^2+u^3)\\
\end{aligned}
$$
and hence
$$
X(u,t) \equiv 4+2u+t+2ut \mod (u^2-1,t^2-1).
$$
Note that in this last expression, the constant term $4$ matches the
total number of orbits.  As an example of the root-of-unity
evaluations predicted by the CSP, note that $X(1,-1)=3$ counting the
elements in the third and fourth orbits, which are fixed by $(1,F)$,
that $X(-1,1)=1$ counting the element in the fourth orbit, which is
fixed by $(\beta,1)$, and that $X(-1,-1)=3$ counting the elements in
the second and fourth orbits, which are fixed by $(\beta,F)$.
\end{example*}

\subsection{Parking functions}

A word of length $\ell$ in the alphabet $\{1,2,\ldots\}$ is called a
{\it parking function} (see, \textit{e.g.}, Haiman \cite{Haiman},
Kung, Sun and Yan \cite{KungSunYan}, Pak and Postnikov
\cite{PakPostnikov}) if the weakly increasing rearrangement $a_1 \leq
a_2 \leq \cdots \leq a_\ell$ of its letters has $a_i \leq i$ for all
$i$.  Theorem~\ref{thm:words-biCSP} turns out to be closely related to
the following result, discussed in Section~\ref{sec:parking-functions}
below.

\begin{theorem}
\label{thm:parking-functions-CSP}
Let $X$ be the set of parking functions of length $\ell$,
permuted by a cyclic group $C$ acting
nearly freely on the set $[\ell]$ of positions.

Let 
$$
X(t):=
\sum_{w \in X} t^{\maj(w)} =
\sum_{w \in X} t^{\inv(w)}.
$$

Then $(X,X(t),C)$ exhibits the CSP.
\end{theorem}

%\vskip.1in
%\noindent
%{\bf Example.}
\begin{example*}
For $\ell=3$, the set $X$ of parking functions is 
\begin{multline*}
X = \{  111,
\quad 112,121,211,
\quad 113,131,311,
\quad 122,212,221,\\
 123,132,213,231,312,321\}.
\end{multline*}
If we compute $X(t)$ term by term using the major index we obtain
\begin{multline*}
X(t)= 
1+
\quad 1+t^2+t+
\quad 1+t^2+t+
\quad 1+t+t^2+\\
 1+t^2+t+t^2+t+t^3 =5 +5t+5t^2+t^3.
\end{multline*}
Taking the cyclic group $C=\ZZ_3$ acting on the positions $[3]$
to be of order $3$, one has
$$
X(t) \equiv 6 +5t+5t^2 \mod t^3-1
$$
whose constant term $6$ counts the total number of $\ZZ_3$-orbits, 
and whose coefficient $5$ on $t^1$ counts the
$5$ free $\ZZ_3$-orbits, namely all but the singleton orbit
$\{111\}$.

Taking the cyclic group $C=\ZZ_2$ acting on the first two positions 
and fixing the third position, one has
$$
X(t) \equiv 10 +6t \mod t^2-1
$$
whose constant term $10$ counts the total number of $\ZZ_2$-orbits,
and whose coefficient of $6$ on $t^1$ counts the $6$
free $\ZZ_2$-orbits, namely those other than the singleton orbits
$$
\{111\}, \{112\},\{113\},\{221\}.
$$ 
\end{example*}

\subsection{Nonnegative matrices}
One can extend the notion of principal specialization to polynomials
$f(\xx,\yy)$ in two sets of variables $\xx=(x_1,\ldots,x_m)$,
$\yy=(y_1,\ldots,y_n)$ by defining their \textit{principal
  $(u,t)$-specialization}
$$
f(1,u,u^2,\ldots,u^{m-1},1,t,t^2,\ldots,t^{n-1}).
$$

\begin{theorem}
\label{thm:matrices-biCSP}
Fix a nonnegative integer $k$, and let $X$ be the collection of all
$m \times n$ matrices $A=(a_{ij})$ with entries in the nonnegative
integers $\NN$ (resp. in $\{0,1\}$) such that $\sum_{ij} a_{ij} = k$.

Let $C_1 \times C_2$ act on $X$,
where $C_1$ and $C_2$ are cyclic groups acting nearly freely on the row indices
$[m]$ and column indices $[n]$.
In the case where $X$ consists of $\{0,1\}$-matrices, make the additional
assumption that $C_1 \times C_2$ is of {\bf odd} order.

Let $X(u,t)$ be the principal $(u,t)$-specialization of
the row-sum and column-sum generating function
$$
\sum_{A =(a_{ij}) \in X} 
\left( \prod_{i,j} (x_i y_j)^{a_{ij}} \right).
$$

Then $(X,X(u,t),C_1 \times C_2)$ exhibits the CSP.
\end{theorem}

%\vskip.1in
%\noindent
%{\bf Example.}
\begin{example*}
Take $k=2$ and $m=n=2$, so that
$X$ consists of $2\times 2$ nonnegative matrices with entries that
sum to $2$.
Then $C_1,=\langle c_1 \rangle$
and  $C_2=\langle c_2 \rangle$ are cyclic groups of order two
that swap the rows and columns, respectively.
The set $X$ decomposes into these four $C_1 \times C_2$-orbits
$$
\left\{
\begin{matrix}
 \left[
 \begin{matrix}
   2 & 0 \\
   0 & 0 
 \end{matrix}
 \right] &
 \left[
 \begin{matrix}
   0 & 2 \\
   0 & 0 
 \end{matrix}
 \right] \\
1 & t^2\\
  & \\
 \left[
 \begin{matrix}
   0 & 0 \\
   2 & 0 
 \end{matrix}
 \right] &
 \left[
 \begin{matrix}
   0 & 0 \\
   0 & 2 
 \end{matrix}
 \right] \\
u^2 & u^2 t^2\\
\end{matrix}
\right\}
, \,\,
\left\{
\begin{matrix}
 \left[
 \begin{matrix}
   1 & 1 \\
   0 & 0 
 \end{matrix}
 \right]  \\
  t \\
  \\
 \left[
 \begin{matrix}
   0 & 0 \\
   1 & 1
 \end{matrix}
 \right] \\
 u^2t 
\end{matrix}
\right\}
, \,\,
\left\{
\begin{matrix}
 \left[
 \begin{matrix}
   1 & 0 \\
   1 & 0 
 \end{matrix}
 \right] &
 \left[
 \begin{matrix}
   0 & 1 \\
   0 & 1 
 \end{matrix}
 \right] \\
u & ut^2 
\end{matrix}
\right\}
, \,\,
\left\{
\begin{matrix}
 \left[
 \begin{matrix}
   1 & 0 \\
   0 & 1 
 \end{matrix}
 \right] &
 \left[
 \begin{matrix}
   0 & 1 \\
   1 & 0 
 \end{matrix}
 \right] \\
ut & ut 
\end{matrix}
\right\}
$$
in which each element of $X$ is shown with the term it contributes to
$X(u,t)$ just below it.  The orbits are arranged tabularly so that $C_1, C_2$
swap bottom-to-top and left-to-right, respectively.  The first orbit
is $C_1 \times C_2$-free, while in the second orbit $c_2$ acts
trivially, in the third orbit $c_1$ acts trivially, and in the
fourth orbit $c_1$ and $c_2$ both swap the two elements of the orbit.

From the data above (or see Section~\ref{sec:hook-formulas} below) one
can calculate
$$
\begin{aligned}
X(u,t) &= (1+u+u^2)(1+t+t^2) +2ut\\
       &\equiv 4+2u+2t+2ut \mod (u^2-1,t^2-1)
\end{aligned}
$$
In this last expression, the constant term $4$ matches the total
number of orbits.  One also has
$$
X(+1,-1)= X(-1,+1)=X(-1,-1)=2
$$
counting the elements in the second, third, fourth orbits
respectively, as they are the elements fixed by $(1,c_2), (c_1,1),
(c_1,c_2)$, respectively.
\end{example*}

\subsection{Graphs}
One of our main CSP constructions will yield the following result
immediately.
\begin{theorem}
\label{thm:graphs-CSP}
Let $X$ be any of the following collections of graphs with $k$ edges
on vertex set $[n]$:
\begin{enumerate}
\item[(i)] graphs allowing multiedges and loops, including the
  possibility of multiple loops on the same vertex,
\item[(ii)] graphs allowing multiedges, but no loops,
\item[(iii)] graphs allowing no multiedges, and  at most one
loop on each vertex,
\item[(iv)] simple graphs, that is, allowing neither multiedges, nor
  loops.
\end{enumerate}
Let $C$ be a cyclic group of acting nearly freely on the vertex set
$[n]$, and thereby permuting the collection of graphs $X$.
Furthermore, in cases (ii),(iii),(iv) make the additional assumption
that $C$ has {\bf odd} order \footnote{For example, depending upon the
  odd/even parity of $n$, one can take $C$ to be generated by an
  $n$-cycle/ $(n-1)$-cycle.}.

Let $X(u)$ be the $u$-principal specialization of the
degree sequence generating function
\begin{equation}
\label{degree-sequence-generating-function}
\sum_{G \in X} \left( \prod_i x_i^{\deg_G(i)} \right)
\end{equation}
where the vertex-degree $\deg_G(i)$ of vertex $i$ counts edges
incident to $i$ with multiplicity, with each loop incident to $i$
contributing $2$ to $\deg_G(i)$.

Then $(X,X(u),C)$ exhibits the CSP.
\end{theorem}

%\vskip.1in
%\noindent
%{\bf Example.}
\begin{example*}
  Take $n=3, k=3$ and consider case (iii) in
  Theorem~\ref{thm:graphs-CSP}, so that $X$ consists of graphs on
  vertex set $[3]$, with $3$ edges total, disallowing multiedges, and
  allowing at most one loop on each vertex.  Let $C$ be a cyclic group
  of order $3$ cycling the vertices.  Note that $C$ has odd order, as
  required in case (iii).

  One can readily check that there are $20$ graphs in $X$, comprising
  six free $C$-orbits, and two singleton orbits: the graph having $3$
  loops, and the triangle graph having no loops.

  One can also calculate that
  $$
  \begin{aligned}
    X(u) &= 2u^3+2u^4+4u^5+4u^6+4u^7+2u^8+2u^9\\
    &\equiv 8+6u+6u^2 \mod (u^3-1)
  \end{aligned}
  $$
  In this last expression, the constant term $8$ counts the total
  number of orbits, and $X(e^{\frac{2\pi i}{3}})=2$ counts the two
  graphs in the singleton orbits, fixed by $C$.
\end{example*}

%%%%%%%%%%%%%%%%%%%%%%%%%%%%%%%%%%%%%%%%%%%%%%%%%%%%%%%%%%%%%%%%%%%
\section{Constructions and proofs}
\label{sec:constructions}

The representation theoretic paradigm for proving a CSP is based on a
simple observation, Proposition~\ref{prop:CSP-as-representations} below 
(\textit{cf.} \cite[\S 2]{RSW}).  

Start with the product $\CCC=C_1 \times \cdots \times C_m$ of finite cyclic
groups.  After picking embeddings of groups $\omega_i: C_i \hookrightarrow \CC^\times$,
note that the irreducible representations of $\CCC$ are
exactly the degree one characters $\omega^\dd:=\prod_i \omega_i^{d_i}$
from \eqref{degree-one-characters}
where $\dd=(d_1,\ldots,d_m)$ satisfies $0 \leq d_i < |C_i|$.

Consequently, given any Laurent polynomial in the variable set $\uu=(u_1,\ldots,u_d)$
\begin{equation}
\label{typical-Laurent-polynomial}
X(\uu) = \sum_{\dd \in \ZZ^m} a_{\dd} \uu^{\dd}
\end{equation}
having nonnegative coefficients $a_{\dd}$, one can construct a $\CCC$-representation
$$
V_{X(\uu)}:=\bigoplus_{\dd \in \ZZ^m} \left( \omega^\dd \right)^{\oplus a_{\dd}}
$$
expressed as an explicit direct sum of the $\CCC$-irreducibles, with
multiplicities.

Given a finite set $X$ permuted by any group $G$,
let $\CC[X]$ denote the associated permutation representation
of $G$, having a $\CC$-vector space basis 
indexed by the elements of $X$.

\begin{proposition}
\label{prop:CSP-as-representations}
Given a finite set $X$ permuted by a finite product of cyclic groups
$\CCC$,  and a Laurent polynomial $X(\uu)$
having nonnegative coefficients,
the triple exhibits the CSP if and only if one has
an isomorphism of
$\CCC$-representations
$$
\CC[X] \cong V_{X(\uu)}.
$$
\end{proposition}
\begin{proof}
The CSP asserts that
every $\cc$ in $\CCC$ has the same
character value on $\CC[X]$, namely $|X^\cc|$,
as its character value on $V_{X(\uu)}$, namely
$\left[ X(\uu) \right]_{u_i=\omega(c_i)}$.
\end{proof}

\subsection{The tensor product construction}

\begin{proposition}
\label{prop:tensor-product-construction}
If both $(X_1,X_1(u),C_1)$ and $(X_2,X_2(u),C_2)$ exhibit the CSP, then so does 
$$
(X_1 \times X_2, \quad X(u,t):=X_1(u)X_2(t), \quad \CCC:=C_1 \times C_2). 
$$ 
\end{proposition}
\begin{proof}
Combine Proposition~\ref{prop:CSP-as-representations} with the
isomorphism of $\CCC$-representations
$$
\CC[X_1 \times X_2] \cong
\CC[X_1] \otimes \CC[X_2].
$$
\end{proof}

\subsection{The symmetric and exterior power constructions}
\label{subsec:powers}

Here we review the
notational tool of plethystic composition of symmetric functions; 
see \cite[Chap. 7 Appendix 2]{Stanley-EC2}.

A {\it symmetric function} $f(x_1,x_2,\ldots)$ on the infinite
variable set $x_1,x_2,\ldots$ with coefficients in a ring $R$ is a
power series in $R[x_1,x_2,\ldots]$ of bounded degree which is
invariant under all permutations of the subscripts on the variables.
One can define the {\it plethystic composition} $f[X(\uu)]$ of such a
symmetric function $f$ with a Laurent polynomial $X(\uu)$ as in
\eqref{typical-Laurent-polynomial} in the following way: if
$n:=\sum_\dd a_\dd=X(1,\ldots,1)$, then one substitutes $x_j=0$ for
all $j > n$ in $f$, and $x_j=u_1^{d_1} \cdots u_m^{d_m}$ for exactly
$a_\dd$ of the variables $x_j$ with $1 \leq j \leq n$.  For example,
if
$$
X(u)=[n]_u=1+u+u^2\cdots+u^{n-1}
$$
then $f[X(u)]=f(1,u,u^2,\ldots,u^{n-1})$ is 
the {\it principal $u$-specialization} described earlier.

We also review the meaning of symmetric functions and plethystic
composition, with regard to the representations of the general linear
group; see \cite[Chap. 7 Appendix 2]{Stanley-EC2}. Let $V$ be a
complex vector space over $\CC$ of dimension $n$, and $GL(V) \cong
GL_n(\CC)$ the general linear group.  One says that a representation
$GL(V) \overset{\rho}{\rightarrow} GL(U)$ is {\it polynomial} if for
some (equivalently, any) choice of $\CC$-bases for $V, U$, the
matrices representing the action of $\rho(g)$ on $U$ have entries
which are polynomial functions in the entries of the matrices
representing the action of $g$ on $V$.  A polynomial representation
$\rho$ gives rise to a symmetric function $f_\rho(x_1,\ldots,x_n)$ in
a finite variable set (with nonnegative integer coefficients) called
its {\it character}, which is the trace of any element $g$ in $GL(V)$
having eigenvalues $x_1,\ldots,x_n$.  One can then interpret the
plethystic composition $f_\rho[X(\uu)]$ for a Laurent polynomial
$X(\uu)$ as in \eqref{typical-Laurent-polynomial} as follows: if $g$
is an element of $GL(V)$ whose eigenvalues are given by the $\uu^\dd$
with eigenvalue multiplicity $a_\dd$, then $\rho(g)$ acts on $U$ with
trace $f_\rho[X(\uu)]$.

The irreducible polynomial representations $GL(V) \overset{S^\lambda}{\rightarrow} GL(U)$
are indexed by partitions $\lambda$ having at most $n=\dim V$ parts.  The
character $f_{S_\lambda}(\xx)$ is the {\it Schur function}
$s_\lambda(\xx)$ in variablex $x_1,\ldots,x_n$, which has a well-known
expression (see, \textit{e.g.}, \cite[\S 7.10]{Stanley-EC2}) 
as a sum over all {\it column-strict tableaux} $P$ with
entries in $[n]=\{1,2,\ldots,n\}$:
\begin{equation}
\label{Schur-function-as-tableaux}
s_\lambda(\xx)=\sum_P \xx^P
\end{equation}
where $\xx^P:=\prod_{i,j}x_{P_{i,j}}$.  Two prototypical examples of
these polynomial representations, are the $k^{th}$ {\it symmetric
  power} $U=\Sym^k(V)$ and the $k^{th}$ {\it exterior power}
$U=\bigwedge^k(V)$, corresponding to the partitions $\lambda=(k)$ and $\lambda=(1^k)$,
respectively. Their characters are
$$
\begin{aligned}
f_{\Sym^k}(\xx)=s_{(k)}(\xx)&=h_k(\xx) 
=\sum_{1 \leq i_1 \leq \cdots \leq i_k \leq n} x_{i_1} \cdots x_{i_k} \\
f_{\bigwedge^k}(\xx)=s_{(1^k)}(\xx) 
&=e_k(\xx) 
=\sum_{1 \leq i_1 < \cdots < i_k \leq n} x_{i_1} \cdots x_{i_k}.
\end{aligned}
$$

We apply these to obtain two more
CSP constructions involving sets and multisets.
For a finite set $X$ and a nonnegative
integer $k$, let 
$$
\bin{X}{k} \quad \text{ and }\multichoose{X}{k}
$$
denote the collection of all $k$-element subsets and $k$-element
multisubsets of $X$.

\begin{proposition}[{\textit{cf.} \cite[Theorem 1.1]{RSW}}]
\label{prop:power-constructions}
If a triple $(X,X(\uu),\CCC)$ exhibits the CSP, then the triple 
$$
\left(\quad  
\multichoose{X}{k}, 
\quad h_k[X(\uu)], 
\quad \CCC \quad
\right)
$$
also exhibits the CSP.

If, in addition, $\CCC$ has {\bf odd} order, then the triple 
$$
\left(\quad  
\bin{X}{k}, 
\quad e_k[X(\uu)], 
\quad \CCC \quad
\right)
$$
also exhibits the CSP.
\end{proposition}

\begin{proof}
  In either case, Proposition~\ref{prop:CSP-as-representations} shows
  that the space $V:=\CC[X]$ has a $\CC$-basis of eigenvectors
  $\{v_\alpha\}$ diagonalizing the action of $\CCC$, in such a way
  that a typical element $\cc=(c_1,\ldots,c_m)$ in $\CCC$ acts with
  eigenvalue $\left[\uu^\dd\right]_{u_i = \omega_i(c_i)}$ on exactly
  $a_\dd$ of the eigenvectors $v_\alpha$.  Thus $\cc$ acts on
  $U=\Sym^k(V)$ and $U=\bigwedge^k(V)$ with traces $h_k[X(\uu)]_{u_i =
    \omega_i(c_i)}$ and $e_k[X(\uu)]_{u_i = \omega_i(c_i)}$,
  respectively.

On the other hand, we claim that one has isomorphisms of
$\CCC$-representations
$$
\begin{aligned}
\Sym^k(V) &\cong \CC\left[ \multichoose{X}{k} \right] \\
{\bigwedge}^k(V) &\cong \CC\left[ \bin{X}{k} \right] \\
\end{aligned}
$$
To see this, note $V=\CC[X]$ has a $\CC$-basis $\{e_x\}_{x \in X}$
permuted in the same way that $\CCC$ acts on $X$.
Therefore $\Sym^k(V)$ and $\bigwedge^k(V)$ have $\CC$-bases of 
monomial symmetric tensors $e_{x_1} \cdots e_{x_k}$ 
and antisymmetric tensors $e_{x_1} \wedge \cdots  \wedge e_{x_k}$.
The group $\CCC$ permutes symmetric tensors
$e_{x_1} \cdots e_{x_k}$ in exactly the way it permutes
$k$-element multisets.
When the group $\CCC$ acts on antisymmetric tensors, it does not
quite act on them in the way that it permutes $k$-element subsets,
but rather permutes and scales them by a sign of $\pm 1$.
However, here one uses the assumption
that $\CCC$ has odd order: when some
$\cc$ in $\CCC$ {\it fixes} some $e_{x_1} \wedge \cdots  \wedge e_{x_k}$
up to sign, meaning that
$$
\cc (e_{x_1} \wedge \cdots  \wedge e_{x_k})
=\pm e_{x_1} \wedge \cdots  \wedge e_{x_k},
$$
then $\cc$ having odd order forces this sign to be $+1$.  
Thus each $\cc$ in $\CCC$ acts with the same trace in
$\bigwedge^k(V)$ as in $\CC\left[ \bin{X}{k} \right].$
\end{proof}

%\vskip.2in
%\noindent
\begin{proof}[Proof of Theorem~\ref{thm:graphs-CSP}]
We will prove a stronger statement.
Assume one has a triple $(X,X(\uu),\CCC)$
exhibiting the CSP.
Proposition~\ref{prop:power-constructions} then shows that the triples
$$
\begin{aligned}
  &\left(\ \multichoose{X}{m}, \quad h_m\left[ X(\uu) \right],\quad \CCC\quad \right) \\
  &\left(\ \bin{X}{m}, \quad e_m\left[ X(\uu) \right],\quad \CCC\quad \right) \\
\end{aligned}
$$
both also exhibit the CSP, assuming that $\CCC$ is of odd order in the
latter case.  Applying Proposition~\ref{prop:power-constructions} one
more time then shows the following result.
\begin{theorem}
\label{thm:hypergraphs-CSP}
Given a triple $(X,X(\uu),\CCC)$ exhibiting the CSP, the following
triples also exhibit the CSP
\begin{enumerate}
\item[(i)]
$$
\left(\quad
\multichoose{\multichoose{X}{m}}{k},\quad
h_k\left[ h_m\left[ X(\uu) \right] \right],\quad
\CCC\quad
\right)
$$
\item[(ii)]
$$
\left(\quad
\multichoose{\bin{X}{m}}{k},\quad
h_k\left[ e_m\left[ X(\uu) \right] \right],\quad
\CCC\quad
\right)
$$
\item[(iii)]
$$
\left(\quad
\bin{\multichoose{X}{m}}{k},\quad
e_k\left[ h_m\left[ X(\uu) \right] \right],\quad
\CCC\quad
\right)
$$
\item[(iv)]
$$
\left(\quad
\bin{\bin{X}{m}}{k},\quad
e_k\left[ e_m\left[ X(\uu) \right] \right],\quad
\CCC\quad
\right)
$$
\end{enumerate}
under the extra assumption that $\CCC$ is of
{\bf odd} order in cases (ii),(iii),(iv).
\end{theorem}
A little reflection shows that 
Theorem~\ref{thm:graphs-CSP} is the
special case of Theorem~\ref{thm:hypergraphs-CSP} in which 
one takes $m=2$, with $X=[n]$ permuted nearly freely by a cyclic group $C$,
and $X(\uu)=[n]_u$. 
\end{proof}
%{\it .}
%\qed

\begin{remark}
  The assumption that $|C|$ is odd in cases (ii),(iii),(iv) of
  Theorem~\ref{thm:graphs-CSP} is perhaps too restrictive; we have not
  made an exhaustive study of the exact hypotheses on $|C|$ and $k$
  which are necessary and sufficient for these triples to exhibit the
  CSP.  However, we do offer one instance of a negative result
  in this regard, as an indication of what one might expect
  (\textit{cf.} \cite[Lemma~2.3]{RSW}).

\begin{proposition}
\label{prop:no hope}
  Let $X$ be the class of graphs with $k$ edges satisfying the
conditions of Theorem~\ref{thm:graphs-CSP}(iii) or (iv),
and let $X(u)$ be as defined there, as the $u$-principal
specialization of \eqref{degree-sequence-generating-function}.

Let $C=\langle c \rangle = \ZZ_n$ with $c=(1,2,\ldots,n)$
cyclically permuting the vertex set $[n]$,
and assume that $n$ is {\bf even}.  

Then the triple $(X,X(u),C)$ exhibits the CSP if and only if
either 
\begin{enumerate}
\item[$\bullet$] $k \in \{ 0,1, \binom{n}{2},\binom{n}{2}-1\}$.
\item[$\bullet$] $n=4, k=3$.
\end{enumerate}
\end{proposition}
\begin{proof}
  The ``if'' direction can be trivially verified. For the ``only if''
  direction, we prove something stronger: For even $n \geq 6$ and $2
  \leq k \leq \binom{n}{2}-2$, \textit{there is no integer $m$ such
    that $u^m X(u)$ gives a triple $(X,u^mX(u),C)$ that exhibits the
    CSP}.

If there were such an integer $m$, then for each $c^d$ in $C$
and $\omega=e^{\frac{2\pi i}{n}}$ one would have 
$$
|X^{c^d}| = | \omega^{dm} X(\omega^d)| = |X(\omega^d)|.
$$
The proof of Theorem~\ref{thm:hypergraphs-CSP} shows
that $X(\omega^d)$ is the trace of $c^d$ acting on
$U=\Sym^k(\bigwedge^2(V))$ or $U=\bigwedge^k(\bigwedge^2(V))$,
in either case.  Since these spaces $U$ have $\CC$-bases
of monomial tensors $v_G$ indexed by graphs $G$, permuted by $C$ up
to root-of-unity scalar multiples, one has
$$
|X(\omega^d)| = \text{ trace of }c^d \text{ on }U
               = \left|\sum_{\substack{G \in X:\\ c^d(G)=G}} 
                       \frac{ c^d(v_G) }{v_G} \right|
               \leq \sum_{\substack{G \in X:\\ c^d(G)=G}} 
                       \left| \frac{ c^d(v_G) }{v_G} \right|
               =|X^{c^d} |
$$
since each of the scalars $\frac{ c^d(v_G) }{v_G}$
is a root-of-unity, with complex modulus $1$.
Furthermore, the case of equality occurs
if and only if these scalars $\frac{c^d(v_G) }{v_G}$ for graphs 
$G$ fixed by $c^d$ are all equal, independent of the choice of $G$.
Thus one can disprove the existence of such a CSP by exhibiting
a choice of $d$ and a choice of two graphs $G, G'$ fixed by
$c^d$ such that $\frac{ c^d(v_G) }{v_G} \neq \frac{ c^d(v_{G'}) }{v_{G'}}$.

When $n=2m$ is even, and $2 \leq k \leq \binom{n}{2}-2$ and $n \geq
6$, one can exhibit such $G,G'$ fixed by $ c^m=(1,m+1)(2,m+2)(3,m+3)
\cdots (m-1,n-1) (m,n) $. Create $G$ by starting with the two edges $
\{1,m+1\},\{2,m+2\} $ and completing $G$ with any $k-2$ other edges
that make $c^{m}(G)=G$.  Then obtain $G'$ from $G$ by replacing the
above two edges with $ \{1,m+2\},\{2,m+1\}, $ and leaving the other
$k-2$ edges of $G$ the same.  One then calculates in $U=
\bigwedge^k(\bigwedge^2(V))$ or $U=\bigwedge^k(\Sym^2(V))$ that
$$
\begin{array}{ll}
c^m\left( (e_1 \wedge e_{m+1}) \wedge (e_2 \wedge e_{m+2}) \right)
&= (e_{m+1} \wedge e_1) \wedge (e_{m+2} \wedge e_2)\\
&= +  (e_1 \wedge e_{m+1}) \wedge (e_2 \wedge e_{m+2})\\
& \\
c^m\left( (e_1 \cdot e_{m+1}) \wedge (e_2 \cdot e_{m+2}) \right)
&= (e_{m+1} \cdot e_1) \wedge (e_{m+2} \cdot e_2)\\
&= +  (e_1 \cdot e_{m+1}) \wedge (e_2 \cdot e_{m+2}),
\end{array}
$$
while,
%& \\
% \text{ while }&   \\
%& \\
$$
\begin{array}{ll}
c^m\left( (e_1 \wedge e_{m+2}) \wedge (e_2 \wedge e_{m+1}) \right)
&= (e_{m+1} \wedge e_2) \wedge (e_{m+2} \wedge e_1)\\
&= -  (e_1 \wedge e_{m+2}) \wedge (e_2 \wedge e_{m+1})\\
& \\
c^m\left( (e_1 \cdot e_{m+2}) \wedge (e_2 \cdot e_{m+1}) \right)
&= (e_{m+1} \cdot e_2) \wedge (e_{m+2} \cdot e_1)\\
&= -  (e_1 \cdot e_{m+2}) \wedge (e_2 \cdot e_{m+1})
\end{array}
$$
Since $G, G'$ share all $k-2$ other edges, we have
$$\frac{ c^m(v_G) }{v_G} =- \frac{ c^m(v_{G'}) }{v_{G'}},$$
in both cases.
\end{proof}

\end{remark}

\vskip.2in
\noindent
{\it Proof of Theorem~\ref{thm:matrices-biCSP}.}
Let $C_1, C_2$ act nearly freely on $[m], [n]$ respectively,
so that the triples
$([m],[m]_u,C_1)$ and $([n],[n]_t,C_2)$
both exhibit the CSP by Proposition~\ref{prop:type-A-regulars}.
Proposition~\ref{prop:tensor-product-construction}
implies 
$$
(\quad 
[m] \times [n], \quad 
[m]_u [n]_t, \quad 
C_1 \times C_2 \quad)
$$
also exhibits the CSP, and
then Proposition~\ref{prop:power-constructions}
shows that
$$
\begin{aligned}
&\left(\quad
\multichoose{[m] \times [n]}{k},\quad
h_k\left[ \,\, [m]_u [n]_t \,\, \right],\quad
C_1 \times C_2\quad
\right) \\
&\left(\quad
\bin{[m] \times [n]}{k},\quad
e_k\left[ \,\, [m]_u [n]_t \,\, \right],\quad
C_1 \times C_2 \quad
\right)
\end{aligned}
$$
also exhibit the CSP, assuming that $C_1 \times C_2$
is of odd order in the latter case.

Now use the usual bijection between
$k$-element subsets (resp., multisubsets) of $[m] \times [n]$
and $m \times n$ matrices $A=(a_{ij})$ whose entries sum to $k$
having $\{0,1\}$ (resp., nonnegative integer) entries:
the entry $a_{ij}$ gives the multiplicity with which
the element $(i,j)$ of $[m] \times [n]$ appears in
the $k$-element subset (resp., multiset).
\qed

\begin{remark}
  It is perhaps worth comparing Theorem~\ref{thm:matrices-biCSP} with
  recent results of Rhoades \cite{Rhoades}.  He again considers a
  subset $X$ of all matrices $A=(a_{ij})$ having nonnegative (resp.,
  $\{0,1\}$) entries.  However his matrices are defined by having
  fixed row and column sum vectors $\mu, \nu$, such that $\mu, \nu$
  are invariant under cyclic groups $C_1,C_2$.  Thus the product $C_1
  \times C_2$ again acts on $X$ by having $C_1 \times C_2$ permute row,
  column indices.  His results \cite[Theorems 1.3, 1.4]{Rhoades}
  describe generating functions $X(u,t)$ for a triple $(X,X(u,t),C_1
  \times C_2)$ exhibiting a CSP in this situation, derived from
  Kostka-Foulkes polynomials, and related to the charge statistics on
  biwords.
\end{remark}

\subsection{The tensor power construction}

Our last construction makes use of two basic representation theoretic
facts:  Schur-Weyl duality in tensor powers $V^{\otimes \ell}$
and the type $A$ case of Springer's theory of regular elements.
We quickly review these here.

\begin{proposition}(Schur-Weyl duality)
Regarding the $\ell$-fold tensor product 
$V^{\otimes \ell}$
as a $GL(V) \times \Symm_\ell$-representation
in which $GL(V)$ acts diagonally and $\Symm_\ell$ permutes the tensor
positions $[\ell]$, one has the following irreducible decomposition:
\begin{equation}
\label{Schur-Weyl-duality}
V^{\otimes \ell} \cong \bigoplus_{\lambda \vdash \ell}
   S^\lambda \otimes \chi^\lambda,
\end{equation}
where $S^\lambda, \chi^\lambda$, respectively, are
the irreducible representations of $GL(V), \Symm_\ell$, respectively,
indexed by $\lambda$.  
\end{proposition}

Springer \cite{Springer} introduced the following
crucial notion.

\noindent
\vskip.1in
{\bf Definition.}
A {\it regular element} in a finite subgroup $W \subset GL(U)$
generated by (complex) reflections is defined to be an element $c$ 
that has an eigenvector lying in $U^{\reg}$,
where $U^{\reg}$ is the complement within $U$ of 
the reflecting hyperplanes for the elements of $W$.
\vskip.1in

Given an irreducible $W$-character $\chi$, the
value $\chi(c)$ on a regular element turns out to be 
determined by the {\it fake-degree polynomial} $f^\chi(t)$,
defined as the polynomial whose coefficient of $t^d$ gives the multiplicity of 
$\chi$ within the $d^{th}$ graded component 
of the {\it coinvariant algebra} $\CC[U]/(\CC[U]^W_+)$;
see \cite[\S 2.5.]{Springer}.

\begin{theorem} \cite[Proposition 4.5]{Springer}
\label{thm:Springer}
Let $c$ be a regular element $c$ in
a finite complex reflection group $W$ acting on $U$, say with eigenvalue $\omega$ on
some eigenvector in $U^{\reg}$.  Then for any $W$-irreducible character $\chi$, 
one has $\chi(c)=f^\chi(\omega^{-1})$
\end{theorem}

\vskip.1in
\noindent
{\bf Example.}
Regarding $W=\Symm_\ell$ as a complex reflection group acting on $U=\CC^\ell$ 
by permuting coordinates, there is a formula 
(due originally to Lusztig; see \cite[Prop. 4.11]{Stanley-invariants}) 
for the fake-degree polynomials
$f^{\lambda}(t)$ associated to the irreducible $\chi^\lambda$,
as a sum over {\it standard Young tableaux} $Q$ of shape $\lambda$
\begin{equation}
\label{fake-degree-as-tableaux}
f^{\lambda}(t) = \sum_Q t^{\maj(Q)}
\end{equation}
in which $\maj(Q)$ is the sum of those entries $i$ in $Q$ for which $i+1$ appears
in a lower row than $i$.
One can also readily check the following:

\begin{proposition}[{\textit{cf.} \cite[\S 5.1]{Springer}}]
\label{prop:type-A-regulars} 
Let $W=\Symm_\ell$ regarded as a complex reflection group  acting on $U=\CC^\ell$
by permuting coordinates.  The following are equivalent for an
element $c$ in $\Symm_\ell$

\begin{enumerate}
\item[(i)] $c$ is regular.
\item[(ii)] $c$ permutes the set of coordinates $[\ell]$ nearly freely.  
\item[(iii)] If $c$ has multiplicative order $d$, then every
primitive $d^{th}$ root of unity is achieved as an eigenvalue for $c$
on at least one eigenvector lying in $U^{\reg}$.  
\item[(iv)] The triple 
$$
(X:=[\ell], \quad X(u):=[\ell]_u, \quad C:=\langle c \rangle)
$$
exhibits the CSP.
\end{enumerate}
In particular, for any embedding $C \overset{\omega}{\rightarrow} \CC^\times$ 
of the cyclic group $C=\langle c \rangle$, such elements $c$
as in (i)-(iv)
satisfy $\chi^\lambda(c)=\left[ f^\lambda(t) \right]_{t=\omega(c)}$.
\end{proposition}

\begin{proposition}
\label{prop:tensor-power-construction}
Let $(X,X(\uu),\CCC)$ be a triple that exhibits the CSP,
and let $C$ be a cyclic group permuting $[\ell]=\{1,2,\ldots,\ell\}$
nearly freely.  
Let $X^\ell$ be the collection of words
of length $\ell$ in the alphabet,
permuted by $\CCC \times C$ in which $\CCC$ acts on the letter values, and 
$C$ acts on the positions $[\ell]$.

Then 
$$
\left(
\quad X^\ell, 
\quad f[X(\uu)](t),
\quad \CCC \times C
\quad
\right)
$$ 
exhibits the CSP, where $f(\xx)(t):=f(\xx,t)$ is the symmetric
function with cofficients in $\ZZ[t]$ appearing in
\eqref{words-maj-inv}.
\end{proposition}
\begin{proof}
Let $V=\CC[X]$, and let $n:=|X|=\dim_\CC V$.
One has an isomorphism of 
$\CCC \times C$-representations
$
\CC[X^\ell] \cong 
V^{\otimes \ell},
$
in which $\CCC$ inherits its action by restriction from the diagonal $GL(V)$-action
on $V^{\otimes \ell}$, and $C$ inherits its action by restriction from the
$\Symm_\ell$ permuting the tensor positions in
$V^{\otimes \ell}$.  Thus Schur-Weyl duality \eqref{Schur-Weyl-duality}
implies that for any element $\xx$ in $GL(V)$ 
having eigenvalues $\xx=(x_1,x_2,\ldots,x_n)$, and any 
element $c$ in $C$, the trace of $(\xx,c)$ acting on
$V^{\otimes \ell}$ will be
\begin{equation}
\label{first-trace-expression}
\sum_{\lambda \vdash \ell} s_\lambda(\xx) \chi^\lambda(c).
\end{equation}
Using Theorem~\ref{thm:Springer} and the tableaux expressions 
\eqref{Schur-function-as-tableaux},\eqref{fake-degree-as-tableaux} 
for $s_\lambda(\xx)$ and for $f^\lambda(t)$, one can rewrite 
\eqref{first-trace-expression} as
\begin{equation}
\left[ \sum_{(P,Q)} \xx^P  t^{\maj(Q)} \right]_{t=\omega(c)}
\end{equation}
in which $(P,Q)$ run through all pairs of Young tableaux
of the same shape, with $P$ column-strict and $Q$ standard.  
Well-known properties of the Robinson-Schensted-Knuth
bijection \cite[\SS 7.11, 7.23]{Stanley-EC2} then let one rewrite this trace of $(\xx,c)$ as
$$
\left[ \sum_{w \in X^\ell} \xx^w  t^{\maj(w)} \right]_{t=\omega(c)}
=\left[ f(\xx,t) \right]_{t=\omega(c).}
$$
Given an element $\cc=(c_1,\ldots,c_m)$ in $\CCC$, 
when it is considered as an element of $GL(V)$,
it has exactly $a_\dd$ of its eigenvalues equal to $\prod_i \omega_i(c_i)$.
Hence the discussion of plethysm in Subsection~\ref{subsec:powers} shows 
the trace of $(\cc,c)$ on $V^{\otimes \ell}$ will be 
$
\left[ f(\xx,t)[X(\uu)] \right]_{u_i=\omega_i(c_i), t=\omega(c)}
$
as desired.
\end{proof}

\begin{proof}[Proof of Theorem~\ref{thm:words-biCSP}]
This is immediate from Propositions \ref{prop:type-A-regulars} and
\ref{prop:tensor-power-construction}.   
\end{proof}

\begin{proof}[Proof of Theorem~\ref{thm:finite-field-biCSP}]
As in the statement of the theorem, consider
$X=\FF_{q^\ell}$ with action of $C_1 \times C_2$ where
$C_1=\FF_q^\times$ acts by scalar multiplication and
$C_2=\Gal(\FF_{q^\ell}/\FF_q)$ acts by powers of
the Frobenius endomorphism $F$.

Since $\FF_{q^\ell}/\FF_q$ is a Galois extension, the Normal Basis
Theorem (see, \textit{e.g.}, Lang \cite[Chap. VIII Theorem
13.1]{Lang}) implies that there exists an element $\alpha \in \FF_{n}$
whose Galois images $\{\alpha, F(\alpha), F^2(\alpha), \ldots,
F^{\ell-1}(\alpha)\}$ give an $\FF_q$-basis for $\FF_{q^\ell}$.  This
choice of basis gives an $\FF_q$-vector space isomorphism $\FF_q^\ell
\rightarrow \FF_{q^\ell}$.  Taking $n=q$, one can precompose this with
a bijection $[n]^\ell \rightarrow \FF_q^\ell$ that comes from
numbering the elements of $\FF_q$ by $[n]$.  The composite is a
bijection $[n]^\ell \rightarrow \FF_{q^\ell}$ which is $C_1 \times
C_2$-equivariant, where the $C_1$-action on the letter values $[n]$ is
nearly free, fixing only the value that labels the zero element of
$\FF_q$, and the $C_2$-action freely permutes the positions in the
words $[n]^\ell$ cyclically.
\end{proof}

\section{Parking functions}
\label{sec:parking-functions}
We prove here something somewhat more general than
Theorem~\ref{thm:parking-functions-CSP}, and
then remark on the relation to Theorem~\ref{thm:parking-functions-CSP}.

\begin{proposition}
\label{prop:rearrangements-CSP}
Let $X$ by any collection of words in $[n]^\ell$ 
which is stable under the action
of $\Symm_{\ell}$ permuting positions, and let
$C$ be a cyclic subgroup of $\Symm_{\ell}$ permuting the positions 
$[\ell]$ nearly freely.

Let 
$
X(t)=\sum_{w \in X} t^{\maj(w)}.
$

Then the triple $(X,X(t),C)$ exhibits the CSP.  
\end{proposition}
\begin{proof}
It suffices to prove this in the special case
where $X$ is the $\Symm_\ell$-orbit of one word $w$.
If $w$ has $k_i$ occurrences of the letter $i$, then
we claim that the $\Symm_\ell$-action on $X$ is
isomorphic to the $\Symm_\ell$-action on flags of nested subsets
$$
\emptyset \subset S_{k_1} \subset S_{k_1+k_2} \subset
S_{k_1+k_2+k_3} \subset \cdots \subset [\ell]
$$
having cardinalities $k_1,k_1+k_2,k_1+k_2+k_3,\ldots$.
This follows because both such $\Symm_\ell$-actions are transitive,
and have the stabilizer of a typical element conjugate to
the Young subgroup 
$\Symm_{k_1} \times \Symm_{k_2} \times \cdots \times \Symm_{k_n}.$

Hence \cite[Proposition 4.4]{RSW} says
that one has a CSP triple $(X,X(q),C)$
where 
$$
X(q)= \qbin{\ell}{k_1, \ldots,k_n}{q}
$$
is the {\it $q$-multinomial coefficient}.  On the other
hand, MacMahon showed that
\[
\qbin{\ell}{m_1, \ldots,m_n}{q}
 = \sum_{w \in X} q^{\maj(w)}.\qedhere
\]
\end{proof}

\begin{remark}
The case of Theorem~\ref{thm:parking-functions-CSP}
in which the cyclic group $C$ permutes the parking functions nearly freely
while fixing the ${\ell}^{th}$ coordinate, so that $C \subset \Symm_{\ell-1}$,
also follows from Theorem~\ref{thm:words-biCSP} by the following reasoning.

Let $\ZZ_{\ell+1}$ denote the integers mod $\ell+1$.
Consider its $\ell$-fold Cartesian product
$\ZZ_{\ell+1}^{\ell}$ with
$\Symm_\ell$ acting by permuting positions.
This descends to an action of 
$\Symm_\ell$ on the quotient group
$
\ZZ_{\ell+1}^\ell/\ZZ \one
$
where $\ZZ \one$ is the diagonal subgroup
generated by $\one:=(1,1,\ldots,1)$.

There are two well-known collections of coset
representatives for this quotient group:
\begin{enumerate}
\item[$\bullet$]
The subgroup isomorphic to $\ZZ_{\ell+1}^{\ell-1}$
consisting of those elements of $\ZZ_{\ell+1}^{\ell}$
having a zero in the $\ell^{th}$-coordinate.

This gives an $\Symm_{\ell-1}$-equivariant bijection
$\ZZ_{\ell+1}^\ell/\ZZ \one \leftrightarrow \ZZ_{\ell+1}^{\ell-1}$.
\item[$\bullet$]
The set $P_\ell$ of all parking functions of length $\ell$;
see Haiman \cite[Proposition 2.6.1]{Haiman}.  

This gives an $\Symm_{\ell}$-equivariant bijection
$\ZZ_{\ell+1}^\ell/\ZZ \one
\leftrightarrow P_\ell$.
\end{enumerate}
Composing these two bijections gives an
$\Symm_{\ell-1}$-equivariant bijection\
$P_\ell  \leftrightarrow \ZZ_{\ell+1}^{\ell-1}$,
and hence also a $C$-equivariant bijection between these
sets.  

Therefore ignoring the action on the values in
Theorem~\ref{thm:words-biCSP} gives this special case
of Theorem~\ref{thm:parking-functions-CSP}.
\end{remark}

\begin{remark}
  Kung, Sun and Yan \cite{KungSunYan} discuss generalizations of
  parking functions, parametrized by the choice of two non-crossing
  lattice paths. By an appropriate choice of the lattices paths, the
  type $A_{\ell-1}$ parking functions $P_\ell$ discussed above, the
  type $B_\ell$ parking functions of Biane \cite{Biane} and Stanley
  \cite{StanleyParking}, and their ``Fuss'' generalizations
  \cite{ArmstrongEu} are seen to be special cases of these parking
  functions.  For every choice of non-crossings lattice paths, the
  associated parking functions are again collections of words which
  are stable under the action of the symmetric group by permuting
  positions. It follows that Proposition~\ref{prop:rearrangements-CSP}
  applies to each of these collections.
\end{remark}

%%%%%%%%%%%%%%%%%%%%%%%%%%%%%%%%%%%%%%%%%%%%%%%%%%%%%%
\section{Hook-length and hook-content formulas}
\label{sec:hook-formulas}
%%%%%%%%%%%%%%%%%%%%%%%%%%%%%%%%%%%%%%%%%%%%%%%%%%%%%%

Many of our CSP theorems have expressed the generating functions as
$X(\uu)=\sum_{x \in X} \uu^{\stat(x)}$ for some statistic(s)
$\stat(x)$ on the set $X$.  We point out here how in most of these
results, there is a more compact expression for $X(\uu)$, because it
is the principal specialization of a symmetric function having an {\it
  explicit} expansions in terms of Schur functions $s_\lambda(\xx)$,
or a Schur function multiplied by fake-degree polynomials
$f^\lambda(t)$.  The latter objects are expressed as convenient
products by the hook-content formula \cite[Section~7.21]{Stanley-EC2}
for principally specialized Schur functions, and the hook formula for
the fake-degree polynomials:
$$
\begin{aligned}
s_\lambda(1,u,u^2,\ldots,u^{n-1})
 &=s_\lambda[\,\, [n]_u \,\,] \\
 &= u^{b(\lambda)} \prod_{x \in \lambda} \frac{1-u^{n+c(x)}}{(1-u^{h(x)})}, \\
f^\lambda(t) 
 & = t^{b(\lambda)} \frac{(t;t)_k}{\prod_{x \in \lambda}(1-t^{h(x)}) }, \\
\end{aligned}
$$
where $x$ runs through each of the $k=|\lambda|$ cells of $\lambda$ in each product,
the hooklength $h(x)$ is the number of cells weakly to the right of $x$
plus the number of cells strictly below it, the content $c(x)$ is
$j-i$ if $x$ lies in row $i$ and column $j$, 
$$
\begin{aligned}
b(\lambda) &:=\sum_i (i-1)\lambda_i=\sum_j \binom{\lambda'_j}{2}\\
(z;t)_k&:=(1-z)(1-zt)(1-zt^2) \cdots (1-zt^{k-1}).
\end{aligned}
$$

\subsection{Matrices}
The Cauchy and dual Cauchy identities \cite[Theorems 7.12.1 and 7.14.3]{Stanley-EC2} 
assert that the generating function
for $m \times n$ nonnegative matrices $A=(a_{ij})$ having entries that sum to $k$
$$
\sum_{A =(a_{ij}) \in X} 
\left( \prod_{i,j} (x_i y_j)^{a_{ij}} \right)
=\sum_{\lambda \vdash k} s_\lambda(\xx) s_\lambda(\yy), \\
$$ 
while for $\{0,1\}$-matrices $A=(a_{ij})$ having entries that sum to $k$
$$
\sum_{A =(a_{ij}) \in X} 
\left( \prod_{i,j} (x_i y_j)^{a_{ij}} \right)
=\sum_{\lambda \vdash k} s_\lambda(\xx) s_{\lambda'}(\yy),
$$
where $\lambda'$ denotes the conjugate or transpose partition
to $\lambda$.  
Consequently, the two generating functions $X(u,t)$ appearing
in Theorem~\ref{thm:matrices-biCSP} have these more compact
expressions:
$$
\begin{aligned}
\sum_{\lambda \vdash k} (ut)^{b(\lambda)} 
\prod_{x \in \lambda} \frac{(1-u^{m+c(x)})(1-t^{n+c(x)})}{(1-u^{h(x)})(1-t^{h(x)})}, \\
\sum_{\lambda \vdash k} (ut)^{b(\lambda)} 
\prod_{x \in \lambda} \frac{(1-u^{m+c(x)})(1-t^{n-c(x)})}{(1-u^{h(x)})(1-t^{h(x)})}.
\end{aligned}
$$

\subsection{Words}

The proof of Proposition~\ref{prop:tensor-power-construction}
shows that $X_{n,\ell}(u,t):=X(u,t)$ appearing in
Theorem~\ref{thm:words-biCSP} on words $[n]^\ell$
is the principal $u$-specialization
of the variables $\xx=(x_1,\ldots,x_n)$ in
$$
\sum_{\lambda \vdash \ell} s_\lambda(\xx) f^\lambda(t)
$$
and hence has the more compact expression
\begin{equation}
\label{compact-word-gf}
X_{n,\ell}(u,t) =(t;t)_\ell \sum_{\lambda \vdash \ell} (ut)^{b(\lambda)}
\prod_{x \in \lambda} \frac{(1-u^{n+c(x)})}{(1-u^{h(x)})(1-t^{h(x)})}.
\end{equation}
We remark on how this implies an interesting reciprocity property of these
polynomials when regarded as functions of $n$.  
%Define $c_{d,\ell}(n;u)$ to be the coefficient of $t^d$ in
%\eqref{compact-word-gf}, so that
%$$
%X_{n,\ell}(u,t) = \sum_{d=0}^{\binom{\ell}{2}} c_{d,\ell}(n,u) t^d.
%$$
%When $u$ is set to $1$, it is not hard to see that $c_{d,\ell}(n,u)$ is
%a polynomial function of the positive integer $n$, and that one has
%a simple symmetry relation among these coefficients:
%\begin{align*}                                                                           %                       
%c_{\binom{\ell}{2}-d,\ell}(n,1) &= (-1)^\ell c_{d,\ell}(-n,1).                         %                                 
%\end{align*}
%However, when the variable $u$ is unspecialized, one can still make
%sense of the coefficients $c_{d,m}(n,u)$ as functions of $u$ that
%involve $n$ in the exponents, and one has the following more general
%symmetry.
\begin{proposition} \label{reciprocity-theorem} We have,
%  c_{\binom{\ell}{2}-d,\ell}(n,u) &= (-u^n)^\ell c_{d,\ell}(-n,u),\text{ or equivalently,}\\
$$ 
t^{\binom{\ell}{2}} X_{n,\ell}(u,t^{-1}) = (-u^n)^\ell \left[ X_{n,\ell}(u,t) \right]_{n \mapsto -n}
$$
\end{proposition}
\begin{proof}Define
  $$
  T_{\lambda,\ell}(n,u,t):= (t;t)_\ell (ut)^{b(\lambda)}
  \prod_{x \in \lambda}\frac{1-u^{n+c(x)}}{(1-u^{h(x)})(1-t^{h(x)})},
  $$
so that $X_{n,\ell}(u,t) = \sum_{\lambda \vdash \ell}  T_{\lambda,\ell}(n,u,t).$
One checks that
  $$
  t^{\binom{\ell}{2}} T_{\lambda,\ell}(n,u,t^{-1}) = (-u^n)^\ell
  T_{\lambda',\ell}(-n,u,t)
  $$
due to the following facts:
  \begin{itemize}
  \item $b(\lambda')-b(\lambda) = \sum_{x \in \lambda} c(x)$,
  \item $b(\lambda')+b(\lambda) + |\lambda|= \sum_{x \in \lambda} h(x)$,
  \item $\lambda,\lambda'$ share the same hook lengths, and
  \item the cells $x$ and $x'$ that correspond under conjugation will
    have opposite contents: $c(x) = -c(x')$.\qedhere
  \end{itemize}
\end{proof}

\subsection{Graphs}

Each of the polynomials $X(u)$ appearing in
Theorem~\ref{thm:hypergraphs-CSP}(i)-(iv) is a specialization of a
plethystic composition of symmetric functions of the form $h_k[h_m],
h_k[e_m], e_k[h_m], e_k[e_m]$.  Thus whenever one knows their explicit
expansion into Schur functions, the principal specialization has a
compact expression.

In fact, such plethysm expansions are known when $m=2$ by various
formulas of Littlewood \cite[Chap I, \S8, Exer. 6]{Macdonald-book},
covering all the cases that appear in
Theorem~\ref{thm:graphs-CSP}(i)-(iv).  Similarly one has such plethysm
expansions whenever $k=2$ (see \cite[Chap I, \S8,
Exer. 9]{Macdonald-book}).

%%%%%%%%%%%%%%%%%%%%%%%%%%%%%%%%%%%%%%%%%%%%%%%%%%%%%%%%%
\appendix
\section{On the well-definition of the CSP}
\label{appendix}

The data implicit in a CSP is more than just a triple
$(X,X(\uu),\CCC)$ of a finite set $X$ with the permutation action of a
finite abelian group $\CCC$, and a polynomial
$X(\uu):=X(u_1,\ldots,u_m)$ in $\ZZ[\uu]$.  Implicitly, one must also
choose two things:
\begin{itemize}
\item[(a)] 
the decomposition $\CCC:=C_1 \times \cdots \times C_m$, and
\item[(b)]
the embeddings of groups 
$\omega_i:C_i \hookrightarrow \CC^\times$
used to phrase the CSP assertion that
$
|X^\cc| = \left[ X(\uu) \right]_{u_i=\omega_i(c_i)}.
$
\end{itemize}
When $m=1$, so that $\CCC$ is a single cyclic group $C=\ZZ_n$,
then it is easy to see that 
$(X,X(u),C)$ exhibits the CSP for some embedding $\omega : C \to
\CC^\times$ if and only if it exhibits the CSP for any other embedding of
$C$ into $\CC^\times$.  This is because there is always an element of
the Galois group $\Gal(\QQ[e^{\frac{2\pi i}{n}}]/\QQ)$ that takes
one such embedding to another, fixing the polynomial $X(u)$. 

On the other hand, the following example shows that for
non-cyclic abelian groups $\CCC$ the embeddings in (b) can make some
difference.

\begin{example}
  Let $X=\ZZ_3 = \{ 1,\omega, \omega^{-1} \} \subset \CC^\times$ where
  $\omega:=e^{\frac{2 \pi i}{3}}$, and let
  $$
  \CCC=C_1 \times C_2 = \ZZ_3 \times \ZZ_3 \subset \CC^\times \times
  \CC^\times
  $$
  act on $X$ via $(\alpha,\beta) \cdot \gamma = \alpha \beta \gamma$
  where here $\alpha, \beta, \gamma$ are all considered inside
  $\CC^\times$.  Then one can check that with respect to the natural
  inclusions $\omega_1, \omega_2$ of $C_1, C_2$ into $\CC^\times$, the
  polynomial $X(u,t)=1+ut+u^2 t^2$ gives a CSP triple
  $(X,X(u,t),\CCC)$.  However, if one alters the embedding of $C_2$ so
  as to send $\beta \mapsto \beta^{-1}$, this is no longer a CSP
  triple.  On the other hand, it can be fixed if one replaces the
  polynomial $X(u,t)$ with the polynomial $X(u,t^2)$.
\end{example}

\begin{example}
  Let $C$ be a cyclic group such that $C = C_1 \times C_2$, where,
  necessarily, $C_1$ and $C_2$ have relatively prime orders. Suppose
  that $(X,X(u),C)$ exhibits the CSP (for some and, hence, any
  embedding $C \to \CC^\times$). It follows that $(X,X(u,t),C_1 \times
  C_2)$ exhibits the CSP, where $X(u,t) = X(ut)$. Indeed, if
  $\omega_i: C_i \to \CC^\times$ are injections for $i=1,2$ then
  $\omega_1 \omega_2 : C \to \CC^\times$ is an injection. It follows
  that
  \[
  X(\omega_1(c_1), \omega_2(c_2)) = X( \omega_1(c_1)\omega_2(c_2)) =
  |X^{(c_1,c_2)}|,
  \]
  which is to say that $(X,X(u,t),C_1 \times C_2)$ exhibits the CSP.
\end{example}
These two examples suggest how one can account for both choices (a)
and (b) in general, by altering the polynomial $X(\uu)$.  
Suppose one is given two decompositions
$$
\CCC:=C_1 \times \cdots \times C_m=C'_1 \times \cdots \times C'_{m'}
$$
and accompanying embeddings
\begin{align*}
  \omega_i &:C_i \hookrightarrow \CC^\times,\qquad i=1,2,\ldots,m,\\
  \omega'_{j} &:C'_{j} \hookrightarrow \CC^\times, \qquad
  j=1,2,\ldots,m'.
\end{align*}
As mentioned in Section~\ref{sec:constructions}, every degree one
character
\[
C_1 \times \dots \times C_m \rightarrow \CC^\times
\]
can be expressed uniquely in the form $\omega^\dd=\omega_1^{d_1}
\cdots \omega_m^{d_m}$ with $0 \leq d_i < |C_i|$ for
$i=1,2,\ldots,m$. It follows that the composite characters defined for $1 \leq j \leq m'$
by
\[
C_1 \times \dots \times C_m = \CCC = C'_1 \times \dots \times C'_{m'}
\stackrel{\pi'_j}{\to} C_j' \stackrel{\omega'_j}{\to} \CC^\times,
\]
where $\pi'_j$ is the projection map, can each be written as
${\omega}^{\dd^{(j)}}$ for some $\dd^{(j)}$.
Given variables $\uu = (u_1,\dots,u_{m})$ we set
\[
\uu^{\dd^{(j)}} = u_1^{d^{(j)}_1} u_2^{d^{(j)}_2} \cdots
u_{m}^{d^{(j)}_{m}}.
\]
One can now check that
\[
(X,\,\, X(v_1,v_2,\dots,v_{m'}), \,\, C_1' \times \dots \times C_{m'}')
\]
exhibits the CSP, with respect to the embeddings $\{\omega_j'\}_{j =
  1,2,\dots, m'}$ if and only if
\[
(X,\,\, X(\uu^{\dd^{(1)}}, \uu^{\dd^{(2)}} , \dots ,\uu^{\dd^{(m')}} ) , \,\,
C_1 \times \dots \times C_m)
\]
exhibits the CSP with respect to to the embeddings
$\{\omega_i\}_{i=1,2,\dots,m}$.

Reversing the roles of $\{\omega_i\}_{i=1,2,\ldots,m}$ and
$\{\omega'_i\}_{i=1,2,\ldots,m'}$, one sees that one can pass between
the polynomials relevant for any two CSPs for $X$ and $\CCC$ by a
simple monomial change-of-variables.

%%%%%%%%%%%%%%%%%%%%%%%%%%%%%%%%%%%%%%%%%%%%%%%%%%%%%%%%%
\section*{Acknowledgements}
The authors thank Dennis Stanton for helpful conversations regarding
this material. They also thank the referees for a careful reading of
the manuscript. 
%%%%%%%%%%%%%%%%%%%%%%%%%%%%%%%%%%%%%%%%%%%%%%%%%%%%%%%%%


\begin{thebibliography}{70}

\bibitem{ArmstrongEu} 
D. Armstrong, S.-P. Eu, 
Nonhomogeneous parking functions and noncrossing partitions,
{\it Electron. J. Combin.} {\bf 15} (2008), \#R146

\bibitem{BarceloReinerStanton}
H. Barcelo, V. Reiner and D. Stanton,
Bimahonian distributions. {\it J. London Math. Soc.} {\bf 77} (2008), 627-646.

\bibitem{Biane}
P. Biane, 
Parking functions of types A and B, {\it Electron. J. Combin.} {\bf
  9(1)} (2002), N7.

\bibitem{FoataSchutzenberger}
D. Foata and M.-P. Sch\"utzenberger,
Major index and inversion number of permutations. {\it Math. Nachr.} {\bf 83}  (1978), 143--159.

\bibitem{Haiman}
M.~D. Haiman, 
Conjectures on the quotient ring by diagonal invariants.  
{\it J. Algebraic Combin.}  {\bf 3}  (1994),  no. 1, 17--76. 


\bibitem{KungSunYan}
J.~P.~S. Kung, X. Sun, and C. Yan, 
Two-boundary lattice paths and parking functions,
{\it Adv. in Appl. Math.} {\bf 39} (2007).

\bibitem{Lang}
S. Lang, 
Algebra, {\it Graduate Texts in Mathematics} {\bf 211}.
Springer-Verlag, New York, 2002.

\bibitem{Macdonald-book}
I.~G. Macdonald,
Symmetric functions and Hall polynomials. Second edition. 
Oxford Mathematical Monographs. Oxford Science Publications. 
The Clarendon Press, Oxford University Press, New York, 1995. 

\bibitem{MoritaNakajima}
H. Morita and T. Nakajima,
A formula of Lascoux-Leclerc-Thibon and representations of
the symmetric groups.
{\it J. Algebraic Combin.} {\bf 24} (2006), 45-60. 

\bibitem{PakPostnikov} I.~Pak, A.~Postnikov, Enumeration of trees and
  one amazing representation of $S_n$ (extended abstract). Proceedings
  of FPSAC'96.

\bibitem{Rhoades}
B. Rhoades,
Hall-Littlewood polynomials and fixed point enumeration.
{\it Disc. Math.} {\bf 310} (2010), 869-876.

\bibitem{RSW}
V.~Reiner, D.~Stanton, and D.~White,
The cyclic sieving phenomenon.
{\it J. Combin. Theory Ser. A} {\bf 108}  (2004), 17--50.

\bibitem{Springer}
T.~A. Springer,
Regular elements of finite reflection groups.
{\it Invent. Math.} {\bf 25} (1974), 159--198.


\bibitem{Stanley-EC1} 
R.~P.~Stanley, 
Enumerative Combinatorics, Volume 1.
{\it Cambridge Studies in Advanced Mathematics} {\bf 49}. 
Cambridge University Press, Cambridge, 1997.

\bibitem{Stanley-EC2} 
R.~P.~Stanley, 
Enumerative Combinatorics, Volume 2.
{\it Cambridge Studies in Advanced Mathematics} {\bf 62}. 
Cambridge University Press, Cambridge, 1999.

\bibitem{Stanley-invariants}
R.~P.~Stanley,
Invariants of finite groups and their applications to combinatorics.
{\it Bull. Amer. Math. Soc. (N.S.)} {\bf 1} (1979), 475--511.

\bibitem{StanleyParking}
R.~P.~Stanley,
Parking functions and noncrossing partitions, {\it Electron. J. Combin.} {\bf 4
(2)} (1997), R20.

\bibitem{Stembridge}
J.~Stembridge,
Some hidden relations involving the ten symmetry classes of plane partitions.  
{\it J. Combin. Theory Ser. A}  {\bf 68}  (1994),  372--409. 

\bibitem{Westbury} 
B.~W.~Westbury, 
Invariant tensors and the cyclic sieving phenomenon.
\href{http://arxiv.org/abs/0912.1512}{\texttt{arXiv:0912.1512}}
  (2009).

\end{thebibliography}
\end{document}